\pdfoutput=1
\RequirePackage{ifpdf}
\ifpdf % We are running pdfTeX in pdf mode
\documentclass[pdftex]{sigma}
\else
\documentclass{sigma}
\fi

\numberwithin{equation}{section}
\newtheorem{Theorem}{Theorem}[section]
\newtheorem{Corollary}[Theorem]{Corollary}
\newtheorem{Lemma}[Theorem]{Lemma}
\newtheorem{Proposition}[Theorem]{Proposition}

\newcommand{\abs}[1]{\lvert #1 \rvert}
\newcommand{\harm}{{\mathrm{ha}}}
\newcommand{\bw}[1]{\textstyle \bigwedge^{#1}}
\newcommand{\crm}{{\mathrm{c}}}

\newcommand{\hrm}{{\mathrm{h}}}
\newcommand{\Hrm}{{\mathrm{H}}}

\newcommand{\Jc}{\mathcal{J}}
\newcommand{\Jrm}{\mathrm{J}}

\newcommand{\Krm}{\mathrm{K}}

\newcommand{\Nat}{\mathbb{N}}
\newcommand{\norm}[1]{\lVert #1 \rVert}
\newcommand{\pd}[2]{\frac{\partial #1}{\partial #2}}
\newcommand{\pdb}[3]{\frac{\partial^2 #1}{\partial #2 \partial #3}}
\newcommand{\pdd}[5]{\frac{\partial^4 #1}{\partial #2 \partial #3 \partial #4 \partial #5}}
\newcommand{\sarm}{\mathrm{sa}}

\newcommand{\Sfk}{\mathfrak{S}}
\newcommand{\sigmatil}{\tilde{\sigma}}
\newcommand{\thetabar}{\bar{\theta}}
\newcommand{\thetatil}{\tilde{\theta}}
\newcommand{\util}{\tilde{u}}
\newcommand{\xtil}{\tilde{x}}

\begin{document}
%\allowdisplaybreaks

\newcommand{\arXivNumber}{1801.06888}

\renewcommand{\PaperNumber}{089}

\FirstPageHeading

\ShortArticleName{On Lagrangians with Reduced-Order Euler--Lagrange Equations}

\ArticleName{On Lagrangians with Reduced-Order\\ Euler--Lagrange Equations}

\Author{David SAUNDERS}

\AuthorNameForHeading{D.~Saunders}

\Address{Department of Mathematics, Faculty of Science, The University of Ostrava,\\ 30.\ dubna 22, 701 03 Ostrava, Czech Republic}
\Email{\href{mailto:david@symplectic.demon.co.uk}{david@symplectic.demon.co.uk}}

\ArticleDates{Received January 26, 2018, in final form August 23, 2018; Published online August 25, 2018}

\Abstract{If a Lagrangian defining a variational problem has order $k$ then its Euler--Lagrange equations generically have order~$2k$. This paper considers the case where the Euler--Lagrange equations have order strictly less than $2k$, and shows that in such a case the Lagrangian must be a polynomial in the highest-order derivative variables, with a specific upper bound on the degree of the polynomial. The paper also provides an explicit formulation, derived from a geometrical construction, of a family of such $k$-th order Lagrangians, and it is conjectured that all such Lagrangians arise in this way.}

\Keywords{Euler--Lagrange equations; reduced-order; projectable}

\Classification{58E30}

\vspace{-2mm}

\section{Introduction}

If $L$ is a Lagrangian function depending on some independent variables $x^i$, some dependent variables $u^\alpha$, and some first derivative variables, then the resulting Euler--Lagrange equations
\begin{gather*}
\pd{L}{u^\alpha} - \frac{{\rm d}}{{\rm d}x^i} \pd{L}{u^\alpha_i} = 0
\end{gather*}
are \looseness=-1 generically of second order: the total derivative operator ${\rm d}/{\rm d}x^i$ maps first-order variables to second-order variables. For a Lagrangian depending also on higher-order derivative variables $u^\alpha_{ij}, u^\alpha_{ijh}, \ldots$ (of order up to~$k$) the Euler--Lagrange equations, written in a multi-index notation as
\begin{gather*}
\sum_{\abs{I}=0}^k (-1)^{\abs{I}} \frac{{\rm d}^{\abs{I}}}{{\rm d}x^I} \pd{L}{u^\alpha_I} = 0 ,
\end{gather*}
are generically of order $2k$. This paper considers the case of $k$-th order Lagrangians whose Euler--Lagrange equations have order strictly less than~$2k$.

The existence of Lagrangians with reduced-order Euler--Lagrange equations has been known for a long time. The Einstein--Hilbert Lagrangian from general relativity (see, for example,~\cite{Car82}) is second-order but its Euler--Lagrange equations, the Einstein field equations, are again second-order rather than fourth-order. In this case, though, the Lagrangian may be written (although not invariantly) as the sum of a~first-order Lagrangian and a~total divergence~\cite[Sections~3.3.1 and~3.3.2]{Car82}, so that these second-order Euler--Lagrange equations may in fact be derived from a~first-order Lagrangian.

Some examples with more substance may be found in~\cite{Krv86,Krv97} in the context of Lagrangians involving a single independent variable. Any such $k$-th order Lagrangian which is linear in the derivative variables of highest order $k$ will give rise to Euler--Lagrange equations of order strictly less than~$2k$. Of course any total derivative $L = {\rm d}f / {\rm d}x$ will satisfy this condition trivially, because its Euler--Lagrange equations will vanish identically. But not every Lagrangian linear in the highest order derivatives is a total derivative (or a total derivative plus a lower-order function); the simplest such example is the first-order Lagrangian $L_1 = uv_x - vu_x$, giving rise to the first-order Euler--Lagrange equations $u_x = v_x = 0$. On the other hand, a Lagrangian involving a single independent variable which is not linear in the derivative variables of highest order~$k$ will necessarily give rise to Euler--Lagrange equations of order~$2k$.

With more than one independent variable, the situation becomes more complicated. It remains the case that linearity in the highest order derivatives is a sufficient condition for reduced-order Euler--Lagrange equations; but now the condition is no longer necessary. For example, in~\cite{PV01} a class of `special Lagrangians' is defined. These are differential forms (integrands of the variational problem) rather than functions, and they are constructed using a procedure of ho\-rizontalization. The coefficient function (the Lagrangian function) is a polynomial of degree $m$ in the derivative variables of highest order $k$, where $m$ is the number of independent variables, and indeed the polynomial is a linear combination of determinants in those variables. Once again, the Euler--Lagrange equations have order strictly less than $2k$. A simple example is the Lagrangian $L_2 = u_x\big(u_{xx} u_{yy} - u_{xy}^2\big)$ which gives rise to a third order Euler--Lagrange equation; the $2$-form $L_2 {\rm d}x \wedge {\rm d}y$ is the horizontalization of $u_x {\rm d}u_x \wedge {\rm d}u_y$.

These special Lagrangians do not, though, exhaust the possibilities. Consider a problem with two independent variables $x$, $y$ and three dependent variables $u$, $v$, $w$. The second order Lagrangian
\begin{gather*}
L_3 = u_{xx} v_{xy} w_{yy} - u_{xx} v_{yy} w_{xy} + u_{xy} v_{yy} w_{xx} - u_{xy} v_{xx} w_{yy} + u_{yy} v_{xx} w_{xy} - u_{yy} v_{xy} w_{xx}
\end{gather*}
gives rise to third order Euler--Lagrange equations but is not a special Lagrangian in the sense of~\cite{PV01} because it is cubic rather than quadratic in the second derivative variables (although it is again a determinant). The same comment applies to the fourth order Lagrangian
\begin{gather*}
L_4 = u_{xxxx} u_{xxyy} u_{yyyy} + 2 u_{xxxy} u_{xxyy} u_{xyyy} - u_{xxxx} u_{xyyy}^2 - u_{xxxy}^2 u_{yyyy} - u_{xxyy}^3,
\end{gather*}
which gives rise to a sixth order Euler--Lagrange equation. This suggests that some more general alternating structure might be needed.

We obtain such a structure by using a version of the theory of differential hyperforms~\cite{Olv82}\footnote{I am grateful to Peter~Olver for providing me with information about this reference.}. These are tensors with symmetry properties corresponding to Young diagrams; if the diagram contains a single column then the tensor is purely alternating and so corresponds to an ordinary differential form. We shall make particular use of hyperforms which are alternating combinations of symmetric tensors. It is known that hyperforms give rise to particular types of determinant known as hyperjacobians (see~\cite{Olv83}), but the determinants used in the present paper appear to be of a somewhat different nature.

The structure of this paper is as follows. In Section~\ref{Snotation} we summarize, for the reader's convenience, the notation and conventions we shall adopt, and in Section~\ref{section3} we prove the polynomial property for Lagrangians with reduced-order Euler--Lagrange equations; this proof makes use (in Lemma~\ref{Lpg}) of a surprising geometrical interpretation of the space of multi-indices. Next in Section~\ref{section4} we give a formal definition of the specific types of hyperform we shall use to construct Lagrangians, and in Section~\ref{Hypsec} we show that any Lagrangian obtained from such a hyperform has reduced-order Euler--Lagrange equations.

Our conjecture is that every Lagrangian with reduced-order Euler--Lagrange equations may be constructed as a sum of Lagrangians obtained in this way, and finally in Section~\ref{section6} we present some evidence in support of that conjecture. It would also be interesting, for future work, to attempt to extend this approach to give geometrical constructions of Lagrangians of order~$k$ whose Euler--Lagrange equations have order less than $2k-1$, $2k-2$, and so on, using the hyperjacobian structure of those equations~\cite[Proposition~4.52]{And89}.

\section{Notation and conventions}\label{Snotation}

We consider a fibred manifold $\pi\colon E \to M$ with $\dim M = m$ and $\dim E = m + n$. Local coordinates on $M$ will be $(x^i)$, and adapted local coordinates on $E$ will be $(x^i,u^\alpha)$. We adopt the convention for wedge products (and also symmetric products) that no fractional factorial coefficient is used, so that for instance
\begin{gather*}
{\rm d}x^i \wedge {\rm d}x^j = {\rm d}x^i \otimes {\rm d}x^j - {\rm d}x^j \otimes {\rm d}x^i , \qquad {\rm d}x^i \odot {\rm d}x^j = {\rm d}x^i \otimes {\rm d}x^j + {\rm d}x^j \otimes {\rm d}x^i
\end{gather*}
without, in these cases, any factor of one-half.

For any order $k \ge 1$ we consider the fibred manifold $\pi_k\colon J^k\pi \to M$ of $k$-th order jets (of local sections of $\pi$) with adapted local coordinates $(x^i,u^\alpha_I)$ where $I \in \Nat^m$ is a multi-index indicating that, for $1 \le i \le m$, $I(i)$ derivatives have been taken with respect to the variable~$x^i$; by default $J^0\pi = E$. We note that $\pi_{k,k-1}\colon J^k\pi \to J^{k-1}\pi$ is an affine bundle with model vector bundle $V\pi \otimes S^k T^* M \to J^{k-1}\pi$, so that it makes sense to say that a function on $J^k\pi$ is polynomial in the `highest derivatives', the fibre coordinates $u^\alpha_I$ where $\abs{I} = \sum\limits_{i=1}^m I(i) = k$. In general our notation will follow that of~\cite{Sau89} except where indicated.

For any $k \ge 0$ we shall let $\Omega^p\big(J^k\pi\big)$ denote the module of $p$-forms on $J^k\pi$. A $p$-form $\omega \in \Omega^p\big(J^k\pi\big)$ is called \emph{horizontal} if the contraction $i_X \omega = 0$ for any vector field $X$ on $J^k\pi$ vertical over $M$; if instead the pullback $(j^k\phi)^* \omega = 0$ for any local section $\phi$ of $\pi$, where $j^k\phi$ denotes the prolonged local section of $\pi_k$, then we say that $\omega$ is a \emph{contact form}. Any $p$-form $\pi_{k,k-1}^* \varpi$, where $\varpi \in \Omega^p\big(J^{k-1}\pi\big)$, may be written uniquely as $\hrm(\varpi) + \crm(\varpi)$ where $\hrm(\varpi)$ is horizontal and $\crm(\varpi)$ is contact; we say that $\hrm(\varpi)$ is the \emph{horizontalization} of~$\varpi$.

A \emph{Lagrangian density} of order $k$ is a horizontal $m$-form $\lambda \in \Omega^m\big(J^k\pi\big)$, and it is \emph{special} in the sense of~\cite{PV01} if $\lambda = \hrm(\varpi)$ for some $\varpi \in \Omega^m\big(J^{k-1}\pi\big)$. Any Lagrangian density $\lambda$, special or not, may be written in coordinates $\lambda = L {\rm d}x^1 \wedge {\rm d}x^2 \wedge \cdots \wedge {\rm d}x^m$ where $L$ is the corresponding local Lagrangian function. The \emph{Euler--Lagrange form} of $\lambda$ is the $(m+1)$-form $\varepsilon$ on~$J^{2k}\pi$ obtained by a canonical procedure from $\lambda$ (essentially taking the exterior derivative and then integrating by parts $k$ times) and incorporates the Euler--Lagrange equations for $\lambda$; in coordinates
\begin{gather*}
\varepsilon = \varepsilon_\alpha {\rm d}u^\alpha \wedge {\rm d}x^1 \wedge {\rm d}x^2 \wedge \cdots \wedge {\rm d}x^m ,
\qquad \varepsilon_\alpha = \sum_{\abs{I}=0}^k (-1)^{\abs{I}} \frac{{\rm d}^{\abs{I}}}{{\rm d}x^I} \pd{L}{u^\alpha_I} .
\end{gather*}
The underlying structures involved in these constructions are either the infinite-order variational bicomplex~\cite{And89} or the finite-order variational sequence~\cite{Kru90}. We shall say that the Euler--Lagrange form~$\varepsilon$ and the associated Euler--Lagrange equations $\varepsilon_\alpha = 0$ are \emph{projectable} if the form $\varepsilon$, generically defined on $J^{2k}\pi$, is projectable to $J^{2k-1}\pi$; the order of the Euler--Lagrange equations will then be strictly less than~$2k$.

Any $p$-form $\Omega^p \big(J^k\pi\big)$ is a section of the bundle $\bw{p}T^* J^k\pi \to J^k\pi$, and any horizontal $p$-form is a section of the pull-back bundle $\bw{p} T^*M \to J^k\pi$. We shall use the terminology \emph{horizontal~$(p,q)$ hyperform of order~$k$} to denote a section of the pullback bundle $\bw{p} S^q T^* M \to J^k\pi$, so that in a chart on $U^k \subset J^k\pi$ such a hyperform looks like
\begin{gather*}
\sum_{\substack{\abs{I_r} = q \\ 1 \le r \le p}}\omega_{I_1 I_2 \cdots I_p} {\rm d}x^{I_1} \wedge {\rm d}x^{I_2} \wedge \cdots \wedge {\rm d}x^{I_p}
\end{gather*}
with $\omega_{I_1 I_2 \cdots I_p} \in C^\infty\big(U^k\big)$, where if the multi-index $I$ corresponds to the list of ordinary indices $i_1 i_2 \cdots i_q$ then
\begin{gather*}
{\rm d}x^I = {\rm d}x^{i_1} \odot {\rm d}x^{i_2} \odot \cdots \odot {\rm d}x^{i_q} .
\end{gather*}
The collection $\Omega^{p,q}_\hrm\big(J^k\pi\big)$ of all such hyperforms of given degree $(p,q)$ is module over both $C^\infty\big(J^k\pi\big)$ and, significantly, over $C^\infty\big(J^{k-1}\pi\big)$.

We now fix the order $k > 0$, and restrict attention to the case $1 \le q \le k$. We shall use italic capitals such as $I$ to denote multi-indices with length $\abs{I} = k-q$ (or, sometimes, with length $\abs{H} = 2k$), calligraphic letters such as $\Jc$ to denote multi-indices with length~$q$, and roman capitals such as $\Krm$ to denote multi-indices with length $k$. Except where stated otherwise we shall adopt the usual summation convention for such multi-indices, as well as for ordinary indices $i,j,\ldots$ and $\alpha,\beta,\ldots$, but readers should be aware that the contraction of symmetric tensors using this multi-index notation introduces numerical constants.

\section{The polynomial property}\label{section3}

Our first result is that if a Lagrangian density of order $k$ gives rise to a projectable Euler--Lagrange form then in any coordinate system the corresponding Lagrangian function must be a polynomial of order at most $p_k$ in the highest order derivative variables $u^\alpha_\Krm$, where $p_k$ is the number of distinct multi-indices $\Krm$ of length $k$.

To illustrate our approach, we describe the procedure for two special cases.

In the first special case we take $k=2$, so that we want to show that $L$ must be a polynomial of degree at most $p_2 = \tfrac{1}{2} m(m+1)$ in the second order derivative variables. We know, as a~consequence of projectability, that the expression
\begin{gather*}
\sum_{\abs{\Krm} = 2} \frac{{\rm d}^2}{{\rm d}x^\Krm} \pd{L}{u^\beta_\Krm}
\end{gather*}
has order strictly less than $4$, so if we expand the second order total derivatives we obtain
\begin{gather*}
\sum_{\abs{\Jrm} = \abs{\Krm} = 2} u^\alpha_{\Jrm+\Krm} \pdb{L}{u^\alpha_\Jrm}{u^\beta_\Krm} + \cdots,
\end{gather*}
where the dots indicate terms whose order is less than $4$. It follows that, for each multi-in\-dex~$H$ with $\abs{H} = 4$, we must have
\begin{gather}\label{Eq4}
\sum_{\Jrm + \Krm = H} \pdb{L}{u^\alpha_\Jrm}{u^\beta_\Krm} = 0 .
\end{gather}
Fix an index $i$ with $1 \le i \le m$, and let $H_i$ be the multi-index satisfying $H_i(i) = 4$, $H_i(j) = 0$ for $j \ne i$; here and more generally we call a multi-index with only a single nonzero entry a~\emph{pure multi-index}. We then see immediately from equation~\eqref{Eq4} that if $\Krm_i$ is the pure multi-index satisfying $\Krm_i(i) = 2$, $\Krm_i(j) = 0$ for $j \ne i$ then
\begin{gather}\label{Eqpure}
\pdb{L}{u^\alpha_{\Krm_i}}{u^\beta_{\Krm_i}} = 0 .
\end{gather}
Now fix indices $i$, $j$ with $j \ne i$, and let $H_{ij}$ be the multi-index satisfying $H_{ij}(i) = H_{ij}(j) = 2$, $H_{ij}(h) = 0$ for $h \ne i,j$; we call this a \emph{mixed multi-index}. If $\Krm_{ij}$ is the mixed multi-index satisfying $\Krm_{ij}(i) = \Krm_{ij}(j) = 1$, $\Krm_{ij}(h) = 0$ for $h \ne i,j$ then we see from~\eqref{Eq4} that
\begin{gather*}
\pdb{L}{u^\alpha_{\Krm_{ij}}}{u^\beta_{\Krm_{ij}}} = - \pdb{L}{u^\alpha_{\Krm_i}}{u^\beta_{\Krm_j}} - \pdb{L}{u^\alpha_{\Krm_j}}{u^\beta_{\Krm_i}} ,
\end{gather*}
so that
\begin{gather}
\pdd{L}{u^\alpha_{\Krm_{ij}}}{u^\beta_{\Krm_{ij}}}{u^\gamma_{\Krm_{ih}}}{u^\delta_{\Krm_{ih}}}
 = \pdd{L}{u^\alpha_{\Krm_i}}{u^\beta_{\Krm_j}}{u^\gamma_{\Krm_i}}{u^\delta_{\Krm_h}}
+ \pdd{L}{u^\alpha_{\Krm_i}}{u^\beta_{\Krm_j}}{u^\gamma_{\Krm_h}}{u^\delta_{\Krm_i}} \nonumber \\
\hphantom{\pdd{L}{u^\alpha_{\Krm_{ij}}}{u^\beta_{\Krm_{ij}}}{u^\gamma_{\Krm_{ih}}}{u^\delta_{\Krm_{ih}}}=}{}
 + \pdd{L}{u^\alpha_{\Krm_j}}{u^\beta_{\Krm_i}}{u^\gamma_{\Krm_i}}{u^\delta_{\Krm_h}}
+ \pdd{L}{u^\alpha_{\Krm_j}}{u^\beta_{\Krm_i}}{u^\gamma_{\Krm_h}}{u^\delta_{\Krm_i}} = 0 .\label{Eqmixed}
\end{gather}
Thus if the expression
\begin{gather*}
\frac{\partial^r L}{\partial u^{\alpha_1}_{\Jrm_1} \partial u^{\alpha_2}_{\Jrm_2} \cdots \partial u^{\alpha_r}_{\Jrm_r}}
\end{gather*}
does not vanish then in the list $(\Jrm_1, \Jrm_2, \ldots, \Jrm_r)$ each distinct pure multi-index $\Krm_i$ can appear at most once (from~\eqref{Eqpure}), and each distinct mixed multi-index $\Krm_{ij}$ can appear at most three times (from~\eqref{Eqmixed} with $h=j$). Furthermore if both $\Krm_{ij}$ and $\Krm_{ih}$ appear then either one or the other must appear only once (from~\eqref{Eqmixed}).

Let $a$ be the number of pure multi-indices $\Krm_i$ in the list, and let $b$, $c$ and $d$ be the number of mixed multi-indices $\Krm_{ij}$ with, respectively, multiplicities $1$, $2$ and $3$. Clearly $b+c+d \le \tfrac{1}{2} m(m-1)$. On the other hand, $a \le m - 2(c+d)$ because if $\Krm_{ij}$ appears with multiplicity $2$ or $3$ (so that if $h \ne i,j$ then $\Krm_{ih}$ and $\Krm_{jh}$ can have multiplicity at most $1$) then neither $\Krm_i$ nor $\Krm_j$ can appear at all. We therefore see that
\begin{gather*}
r = a + b + 2c + 3d \le m - 2(c+d) + \tfrac{1}{2} m(m-1) + c + 2d \le \tfrac{1}{2} m(m+1) = p_2 .
\end{gather*}

In that first special case with $k=2$ we were able to see explicitly the polynomial structure of the Lagrangian, but for higher orders this detailed investigation rapidly becomes unmanageable, so we need to adopt a more abstract approach. For our second special case we therefore let $k$ be arbitrary, but take $m=2$. There are now $p_k = k+1$ multi-indices of length $k$, and now the consequence of projectability is that, for $0 \le h \le k$,
\begin{gather}\label{Eq5}
\sum_{l=0}^{k-h} \pdb{L}{u^\alpha_{(k-l,l)}}{u^\beta_{(h+l,k-h-l)}} = 0.
\end{gather}
As before we use this relation to manipulate the repeated partial derivatives of $L$, but now we need a mechanism to keep track of what we are doing and help us avoid going round in circles. We do this by defining the \emph{weight}\footnote{Note that this type of weight is different from the system of weights defined in a similar context in~\cite[De\-fi\-ni\-tion~4.18]{And89}.} of a multi-index $\Jrm$ of length $k$ to be the squared Euclidean norm, $\norm{\Jrm}^2 = \sum\limits_{j=1}^m (\Jrm(j))^2$. We also define the weight of a list of multi-indices $(\Jrm_1, \Jrm_2, \ldots, \Jrm_r)$ to be the sum of the individual weights $\norm{\Jrm_1}^2 + \norm{\Jrm_2}^2 + \cdots + \norm{\Jrm_r}^2$. It is important to note that the maximum weight of a multi-index is $k^2$, and that this maximum is achieved if, and only if, the multi-index is pure. In our special case with $m=2$ we have $\norm{(h,l)}^2 = h^2 + l^2$, and the maximum weight is achieved by $(k,0)$ and $(0,k)$.

We now use this to show that every repeated derivative of order $p_k+1 = k+2$,
\begin{gather}\label{RD}
\frac{\partial^{k+2} L}{\partial u^{\alpha_1}_{\Jrm_1} \partial u^{\alpha_2}_{\Jrm_2} \cdots \partial u^{\alpha_{k+2}}_{\Jrm_{k+2}}}
\end{gather}
must vanish. Here we use the fact that, as in the case $k=2$, we have{\samepage
\begin{gather*}
\sum_{\Jrm + \Krm = H} \pdb{L}{u^\alpha_\Jrm}{u^\beta_\Krm} = 0,
\end{gather*}
where $\abs{H}=2k$, so that if $\Jrm$ is a pure multi-index then $\partial^2 L / \partial u^\alpha_\Jrm \partial u^\beta_\Jrm = 0$.}

Let $(\Jrm_1, \Jrm_2, \ldots, \Jrm_{k+2})$ be the list of multi-indices in the repeated derivative, so that necessarily at least two of these multi-indices must be equal. If they are both equal to $(k,0)$, or both equal to $(0,k)$, so that they are pure, then the repeated derivative must vanish. So suppose this is not the case, and assume without loss of generality that $\Jrm_1 = \Jrm_2 = (k-g,g)$. From~\eqref{Eq5} with $h=k-2g$ we see that
\begin{align*}
\pdb{L}{u^{\alpha_1}_{(k-g,g)}}{u^{\alpha_2}_{(k-g,g)}}
= \sum_{\substack{\Krm_1 + \Krm_2 = (2k-2g,2g) \\ \Krm_1, \Krm_2 \ne (k-g,g)}} - \pdb{L}{u^{\alpha_1}_{\Krm_1}}{u^{\alpha_2}_{\Krm_2}} .
\end{align*}
On the left-hand side the weight of the two multi-indices $(\Jrm_1,\Jrm_2)$ is $2 \bigl( (k-g)^2 + g^2 \bigr)$, whereas in a term on the right-hand side with $l \ne 0$ and $\Krm_1 = (k-g-l,g+l)$, $\Krm_2 = (k-g+l,g-l)$ the weight of $(\Krm_1,\Krm_2)$ is
\begin{gather*}
 \bigl( (k-g-l)^2 + (g+l)^2 \bigr) + \bigl( (k-g+l)^2 + (g-l)^2 \bigr) = 2 \bigl( (k-g)^2 + g^2 \bigr) + 4l^2 .
\end{gather*}
We may therefore write our original repeated derivative~\eqref{RD} as (apart from an overall sign) a~sum of similar repeated derivatives where, in each term, the weight of the multi-index list has increased. Furthermore, each new term also has the property that two of its multi-indices must be equal, so by repeating the process we must eventually be able to write~\eqref{RD} as a sum or difference of terms, each of which has three pure multi-indices of maximum weight~$k^2$ (that is, either $(k,0)$ or $(0,k)$), so that two of the pure multi-indices must be equal, and therefore each such term must vanish.

The proof of the general result with $k$ and $m$ both arbitrary follows essentially the same approach as in the second special case. We first confirm the relationship between the weights of multi-indices of length~$k$.
\begin{Lemma}\label{Lpg} If $\abs{\Jrm} = \abs{\Krm_1} = \abs{\Krm_2} = k$ and $2\Jrm = \Krm_1 + \Krm_2$ then the weight $\norm{J}$ satisfies $2\norm{\Jrm}^2 \le \norm{\Krm_1}^2 + \norm{\Krm_2}^2$, with equality only when $\Krm_1 = \Krm_2 = \Jrm$.
\end{Lemma}
\begin{proof}
This is just the parallelogram rule for any Euclidean norm, that
\begin{gather*}
2 \norm{x}^2 \le 2 \norm{x}^2 + 2 \norm{y}^2 = \norm{x+y}^2 + \norm{x-y}^2
\end{gather*}
with equality only when $y=0$.
\end{proof}

\begin{Theorem}\label{poly}If the Lagrangian density $\lambda$ on $J^k\pi$ has projectable Euler--Lagrange equations then in any coordinate system $\lambda = L {\rm d}x^1 \wedge {\rm d}x^2 \wedge \cdots \wedge {\rm d}x^m$ where the function~$L$, defined locally on~$J^k\pi$, is a polynomial of order at most $p_k$ in the highest order derivative variables $u^\alpha_\Jrm$, where~$p_k$ is the number of distinct multi-indices of length $k$.
\end{Theorem}
\begin{proof}The consequence of projectability is now that the expression
\begin{gather*}
\sum_{\abs{\Krm} = k} \frac{{\rm d}^{\abs{\Krm}}}{{\rm d}x^\Krm} \pd{L}{u^\beta_\Krm}
\end{gather*}
has order strictly less than $2k$. Expanding the $k$-th order total derivatives gives
\begin{gather*}
\sum_{\abs{\Jrm} = \abs{\Krm} = k} u^\alpha_{\Jrm+\Krm} \pdb{L}{u^\alpha_\Jrm}{u^\beta_\Krm} + \cdots,
\end{gather*}
where the dots indicate terms whose order is less than $2k$. It follows that, for each multi-index~$H$ with $\abs{H} = 2k$, we must have
\begin{gather}\label{Eq6}
\sum_{\Jrm + \Krm = H} \pdb{L}{u^\alpha_\Jrm}{u^\beta_\Krm} = 0 .
\end{gather}
Now consider the repeated derivative of order $p+1$
\begin{gather}\label{RD2}
\frac{\partial^r L}{\partial u^{\alpha_1}_{\Jrm_1} \partial u^{\alpha_2}_{\Jrm_2} \cdots \partial u^{\alpha_{p+1}}_{\Jrm_{p+1}}},
\end{gather}
where $\abs{\Jrm_1} = \abs{\Jrm_2} = \cdots = \abs{\Jrm_{p+1}} = k$, and suppose that in the list of multi-indices $(\Jrm_1, \Jrm_2, {\ldots}, \Jrm_{p+1})$ we have $\Jrm_r = \Jrm_s = \Jrm$. Use~\eqref{Eq6} to write
\begin{gather*}
\pdb{L}{u^{\alpha_r}_\Jrm}{u^{\alpha_s}_\Jrm} = \sum_{\substack{\Krm_1 + \Krm_2 = 2\Jrm \\ \Krm_1, \Krm_2 \ne \Jrm}} - \pdb{L}{u^{\alpha_r}_{\Krm_1}}{u^{\alpha_s}_{\Krm_2}},
\end{gather*}
so that by Lemma~\ref{Lpg} we have in each term on the right-hand side $\norm{\Krm_1}^2 + \norm{\Krm_2}^2 > 2 \norm{\Jrm}^2$. By repeating this process we must eventually be able to write~\eqref{RD2} as a sum or difference of terms, each of which has $m+1$ pure multi-indices of maximum weight $k^2$ (so that two of its pure multi-indices must be equal) and therefore each of which must vanish.
\end{proof}

\section{Hyperforms}\label{section4}

The necessary condition given above for a Lagrangian density $\lambda$ on $J^k\pi$ to have projectable Euler--Lagrange equations is obviously not sufficient; but the requirement that~$\lambda$ be the horizontalization of some $m$-form on $J^{k-1}\pi$ is, as noted in the Introduction, too strong. We shall define a weaker condition on~$\lambda$ which will still be sufficient to ensure that the Euler--Lagrange equations are projectable, using the idea of a horizontal $(p,q)$ hyperform introduced in Section~\ref{Snotation}.

We first consider horizontal $(1,q)$ hyperforms; any such hyperform $\theta$ may be written in coordinates on $U^k$ as $\theta_\Jc {\rm d}x^\Jc$ where $\theta_\Jc \in C^\infty\big(U^k\big)$. We have mentioned that $\pi_{k,k-1}\colon J^k\pi \to J^{k-1}\pi$ is an affine bundle so that for any point of $J^{k-1}\pi$ the fibre over that point is an affine space. The restriction of $\theta$ to that fibre takes its values in the appropriate fibre of $S^q T^* M$, a~vector space, so it makes sense to ask whether that restriction is an affine map. If this is the case for every fibre of $\pi_{k-1,k}$ then we say that $\theta$ is an \emph{affine $(1,q)$ hyperform}; the coordinate representation of $\theta$ then becomes
\begin{gather*}
\big( \theta^\Krm_{\alpha\Jc} u^\alpha_\Krm + \theta_\Jc \big) {\rm d}x^\Jc,
\end{gather*}
where now $\theta^\Krm_{\alpha\Jc}, \theta_\Jc \in C^\infty(U^{k-1})$, $U^{k-1} \subset J^{k-1}\pi$. (Recall here that the roman multi-indices~$\Krm$ have length~$k$, whereas the calligraphic multi-indices $\Jc$ have length~$q$.)

The collection of affine hyperforms is, however, too large for our purposes. To obtain a suitable restriction, we note that the map $\theta \colon J^k\pi \to S^q T^* M$ is affine precisely when the associated difference map $\thetabar \colon V\pi \otimes_{J^{k-1}\pi} S^k T^* M \to S^q T^* M$ is linear on each fibre over $J^{k-1}\pi$. We say that $\theta$ is a \emph{special affine hyperform} if there is a tensor $\thetatil \in V\pi^* \otimes_{J^{k-1}\pi} S^{k-q} TM$ such that the fibre-linear map $\thetabar$ is given by the contraction of elements of $V\pi \otimes_{J^{k-1}\pi} S^k T^* M$ with $\thetatil$. We shall let $\Omega^{1,q}_\sarm\big(J^k\pi\big)$ denote the collection of such special affine hyperforms; it is a module over $C^\infty\big(J^{k-1}\pi\big)$, though not of course over $C^\infty\big(J^k\pi\big)$. A special affine hyperform may therefore be written in coordinates as
\begin{gather*}
\bigl( \theta^I_\alpha u^\alpha_{I+\Jc} + \theta_\Jc \bigr) {\rm d}x^\Jc,
\end{gather*}
where $\theta^I_\alpha, \theta_\Jc \in C^\infty\big(U^{k-1}\big)$, $U^{k-1} \subset J^{k-1}\pi$. (Here the italic multi-indices $I$ have length $k-q$.)

As examples of affine and special affine hyperforms, consider the case where $m=2$ and $n=1$, with coordinates $x$, $y$, $u$, and where $q = k = 2$, so that in this case each multi-index $I$ is zero. An affine hyperform will have a coordinate representation
\begin{gather*}
\bigl( \theta^{xx}_{xx} u_{xx} + \theta^{xy}_{xx} u_{xy} + \theta^{yy}_{xx} u_{yy} + \theta_{xx} \bigr) {\rm d}x \odot {\rm d}x+ \bigl( \theta^{xx}_{xy} u_{xx} + \theta^{xy}_{xy} u_{xy} + \theta^{yy}_{xy} u_{yy} + \theta_{xy }\bigr) {\rm d}x \odot {\rm d}y \\
\qquad {}
+ \bigl( \theta^{xx}_{yy} u_{yy} + \theta^{xy}_{yy} u_{xy} + \theta^{yy}_{yy} u_{yy} + \theta_{yy} \bigr) {\rm d}y \odot {\rm d}y ,
\end{gather*}
where the functions $\theta^{xx}_{xx}, \theta^{xy}_{xx}, \ldots$ are at most first order, whereas a special affine hyperform will have a~coordinate representation
\begin{gather*}
\bigl( \theta_0 u_{xx} + \theta_{xx} \bigr) {\rm d}x \odot {\rm d}x + \bigl( \theta_0 u_{xy} + \theta_{xy }\bigr) {\rm d}x \odot {\rm d}y
+ \bigl( \theta_0 u_{yy} + \theta_{yy} \bigr) {\rm d}y \odot {\rm d}y ,
\end{gather*}
where each term involves only a single second order coordinate, and where the (at most first order) function $\theta_0$ is the same for all three terms.

We may see the relationship between this definition and the operation of horizontalization on ordinary $1$-forms by considering the special case where $q = 1$. In this case a special affine hyperform $\theta$ (now just a horizontal $1$-form) may be written in coordinates as $\big(\theta^I_\alpha u^\alpha_{I+1_i} + \theta_i\big) {\rm d}x^i$ (where $1_i$ denotes the multi-index with a~single~$1$ in position~$i$) and is the horizontalization of, for instance, the $1$-form $\theta^I_\alpha {\rm d}u^\alpha_I + \theta_i {\rm d}x^i$. There is, however, no well-defined horizontalization operator mapping forms to hyperforms when $q \ge 2$.

We now consider horizontal $(p,q)$ hyperforms, where $p$ is fixed to equal the number $p_q$ of distinct multi-indices $I \in \Nat^m$ of length~$q$, so that
\begin{gather*}
p_q = \begin{pmatrix} m+q-1 \\ q \end{pmatrix} = \frac{(m+q-1)!}{q! (m-1)!} ;
\end{gather*}
the fibre dimension of $S^q T^* M$ is then equal to $p_q$, so that $\bw{p_q} S^q T^* M \to J^k\pi$ is a line bundle. We shall say that such a section of this bundle, a horizontal $(p_q,q)$ hyperform $\omega$, is \emph{hyperaffine} if it can be written as a sum of wedge products of special affine $(1,q)$ hyperforms, and we shall let $\Omega^{p_q,q}_\harm\big(J^k\pi\big)$ denote the collection of such hyperforms; again this is a module over $C^\infty\big(J^{k-1}\pi\big)$. If for a single wedge product
\begin{gather*}
\omega = \theta_1 \wedge \theta_2 \wedge \cdots \wedge \theta_{p_q},
\end{gather*}
where
\begin{gather*}
\theta_r \in \Omega^{1,q}_\harm\big(J^k\pi\big) , \qquad \theta_r = \bigl( \theta^I_{r,\alpha} u^\alpha_{I+\Jc} + \theta_{r,\Jc} \bigr) {\rm d}x^\Jc,
\end{gather*}
then in the coordinate expression for $\omega$ the coefficient of the single local basis element ${\rm d}x^{\Jc_1} \wedge {\rm d}x^{\Jc_2} \wedge \cdots \wedge {\rm d}x^{\Jc_{p_q}}$ will be given as a linear combination (by functions projectable to~$J^{k-1}\pi$) of determinants in the highest order derivative variables $u^\alpha_{I+\Jc}$, ranging in size up to $(p_q \times p_q)$; for instance the largest determinant will take the form
\begin{gather}\label{Det}
\begin{vmatrix}
u^{\alpha_1}_{I_1+\Jc_1} & u^{\alpha_1}_{I_1+\Jc_2} & \cdots & u^{\alpha_1}_{I_1+\Jc_{p_q}} \\
u^{\alpha_2}_{I_2+\Jc_1} & u^{\alpha_2}_{I_2+\Jc_2} & \cdots & u^{\alpha_2}_{I_2+\Jc_{p_q}} \\
\vdots & \vdots & \ddots & \vdots \\
u^{\alpha_{p_q}}_{I_{p_q}+\Jc_1} & u^{\alpha_{p_q}}_{I_{p_q}+\Jc_2} & \cdots & u^{\alpha_{p_q}}_{I_{p_q}+\Jc_{p_q}}
\end{vmatrix},
\end{gather}
where $u^{\alpha_1}_{I_1}, u^{\alpha_2}_{I_2}, \ldots, u^{\alpha_{p_q}}_{I_{p_q}}$ are derivative variables of order $k-q$, and smaller determinants, arising when one or more of the $(1,q)$ hyperforms $\theta_r$ is projectable to~$J^l\pi$ with $l < k$, will be obtained as suitably-sized minors. (If those derivative variables are not distinct then the largest determinant will vanish, and this always happens when $n < p_q$. It is nevertheless the case that sufficiently small minors will be nonzero.)

As examples, we may see that the four Lagrangian functions mentioned in the Introduction all arise as such coefficients. For $L_1$ we take $q=p_q=k=1$ and for $L_2$ we take $q=1$, $p_q=2$, $k=2$; both the corresponding Lagrangian densities arise from conventional horizontalization. For $L_3$ we take $q=k=2$, $p_q=3$, the variables $u$, $v$, $w$ and the hyperform obtained from the wedge product
\begin{gather*}
\bigl( u_{xx} {\rm d}x \odot {\rm d}x + u_{xy} {\rm d}x \odot {\rm d}y + u_{yy} {\rm d}y \odot {\rm d}y \bigr)
\wedge \bigl( v_{xx} {\rm d}x \odot {\rm d}x + v_{xy} {\rm d}x \odot {\rm d}y + v_{yy} {\rm d}y \odot {\rm d}y \bigr) \\
\qquad{} \wedge \bigl( w_{xx} {\rm d}x \odot dx + w_{xy} {\rm d}x \odot {\rm d}y + w_{yy} {\rm d}y \odot {\rm d}y \bigr),
\end{gather*}
so that
\begin{gather*}
L_3 = \begin{vmatrix}
u_{xx} & u_{xy} & u_{yy} \\ v_{xx} & v_{xy} & v_{yy} \\ w_{xx} & w_{xy} & w_{yy}
\end{vmatrix} ;
\end{gather*}
for $L_4$ we take $q=2$, $p_q=3$, $k=4$, the variables $u_{xx}$, $u_{xy}$, $u_{yy}$ and the hyperform obtained from the wedge product
\begin{gather*}
\bigl( u_{xxxx} {\rm d}x \odot {\rm d}x + u_{xxxy} {\rm d}x \odot {\rm d}y + u_{xxyy} {\rm d}y \odot {\rm d}y \bigr) \\
\qquad{}\wedge \bigl( u_{xxxy} {\rm d}x \odot {\rm d}x + u_{xxyy} {\rm d}x \odot {\rm d}y + u_{xyyy} {\rm d}y \odot {\rm d}y \bigr) \\
\qquad{}\wedge \bigl( u_{xxyy} {\rm d}x \odot {\rm d}x + u_{xyyy} {\rm d}x \odot {\rm d}y + u_{yyyy} {\rm d}y \odot {\rm d}y \bigr),
\end{gather*}
so that
\begin{gather*}
L_4 = \begin{vmatrix}
u_{xxxx} & u_{xxxy} & u_{xxyy} \\ u_{xxxy} & u_{xxyy} & u_{xyyy} \\ u_{xxyy} & u_{xyyy} & u_{yyyy}
\end{vmatrix} .
\end{gather*}

\section{Hyperaffine Lagrangians}\label{Hypsec}

The examples at the end of the previous section suggest how we might relate the construction of hyperaffine $(p_q,q)$ hyperforms to Lagrangian densities. We note that such a hyperform may be written in coordinates as
\begin{gather*}
\omega = \omega_q {\rm d}x^{\Jc_1} \wedge {\rm d}x^{\Jc_2} \wedge \cdots \wedge {\rm d}x^{\Jc_{p_q}} ,
\end{gather*}
and when $q > 1$ then this is obviously different from an ordinary horizontal $m$-form such as $\lambda$. The two types of object are, nevertheless, related: they are both sections of line bundles, and in coordinates each has a single coefficient function, $L$ or $\omega_q$. Furthermore, under a change of coordinates $(x^i,u^\alpha) \mapsto (\xtil^i,\util^\alpha)$, $L$ is altered by the Jacobian of the transformation $x^i \mapsto \xtil^i$, whereas $\omega_q$ is altered by a power of that Jacobian. As the condition for~$\omega$ to be hyperaffine may be expressed in terms of $\omega_q$ in a way which is independent of transformations of the base coordinates~$x^i$, it makes sense to say that a Lagrangian density $\lambda$ is hyperaffine if, in any coordinate system, the corresponding local Lagrangian function $L$ may be written as a sum $L = \sum\limits_{q=1}^k \omega_q$ where each $\omega_q$ is the coefficient in that coordinate system of a hyperaffine $(p_q,q)$ hyperform with $p_q = (m+q-1)!/q!(m-1)!$.

For example, in the case where $m=2$ and $n=1$ with coordinates $x$, $y$, $u$, we might consider the third-order Lagrangian function
\begin{gather*}
L = u_{xxx} u_{yyy} - u_{xxy} u_{xyy} .
\end{gather*}
We may write $L$ as $\omega_2$, the scalar coefficient of $\omega = \theta_1 \wedge \theta_2 \wedge \theta_3$, where
\begin{gather*}
\theta_1 = u_{xxx} {\rm d}x \odot {\rm d}x + u_{xxy} {\rm d}x \odot {\rm d}y + u_{xyy} {\rm d}y \odot {\rm d}y, \\
\theta_2 = u_{xxy} {\rm d}x \odot {\rm d}x + u_{xyy} {\rm d}x \odot {\rm d}y + u_{yyy} {\rm d}y \odot {\rm d}y, \\
\theta_3 = {\rm d}x \odot {\rm d}y,
\end{gather*}
so that $\omega_2$ is a (non-vanishing) $2 \times 2$ minor of the (vanishing) determinant
\begin{gather*}
\begin{vmatrix}
u_{xxx} & u_{xxy} & u_{xyy} \\ u_{xxy} & u_{xyy} & u_{yyy} \\ u_{xxy} & u_{xyy} & u_{yyy}
\end{vmatrix} .
\end{gather*}
We see that $L$ is a null Lagrangian, so its Euler--Lagrange equations are trivially projectable. Indeed the significance of our definition comes from the following result.

\begin{Theorem}\label{Th1}If $\lambda$ is a hyperaffine Lagrangian density on $J^k\pi$ then $\lambda$ has projectable Euler--Lagrange equations.
\end{Theorem}
\begin{proof}It is sufficient to prove that a function $L$ given in coordinates as an $h \times h$ minor of the determinant~\eqref{Det} gives rise to Euler--Lagrange equations of order strictly less than~$2k$. As terms in those equations of order $2k$ can arise only when considering
\begin{gather*}
\sum_{\abs{\Krm} = k} \frac{{\rm d}^{\abs{\Krm}}}{{\rm d}x^\Krm} \pd{L}{u^\beta_\Krm} ,
\end{gather*}
where we have written the sum over the multi-indices $\Krm$ explicitly, it is sufficient to show that each such term (for a given index $\beta$) vanishes when $L$ is such a determinant.

Write $L$ in the form
\begin{gather*}
L = \sum_{\sigma \in \Sfk_h} \varepsilon_\sigma u^{\alpha_1}_{I_1 + \Jc_{\sigma(1)}} u^{\alpha_2}_{I_2 + \Jc_{\sigma(2)}} \cdots u^{\alpha_h}_{I_h + \Jc_{\sigma(h)}},
\end{gather*}
where $\Sfk_h$ is the permutation group and $\varepsilon_\sigma = \pm 1$ is the parity of the permutation $\sigma$; then for any given multi-index $\Krm$ we have
\begin{gather*}
\frac{{\rm d}^{\abs{\Krm}}}{{\rm d}x^\Krm} \pd{L}{u^\beta_\Krm}
= \sum_{\substack{1 \le r,s \le h \\ s \ne r}} \sum_{\sigma \in \Sfk_h} \delta^{\alpha_r}_\beta \delta^K_{I_r + \Jc_{\sigma(r)}}
\varepsilon_\sigma \Phi_{rs\sigma} u^{\alpha_s}_{I_r + I_s + \Jc_{\sigma(r)} + \Jc_{\sigma(s)}},
\end{gather*}
where the coefficient functions $\Phi_{rs\sigma}$ are given by
\begin{gather*}
\Phi_{rs\sigma} = u^{\alpha_1}_{I_1 + \Jc_{\sigma(1)}} u^{\alpha_2}_{I_2 + \Jc_{\sigma(2)}} \cdots \widehat{r} \cdots \widehat{s} \cdots u^{\alpha_h}_{I_h + \Jc_{\sigma(h)}}
\end{gather*}
with the circumflex denoting the omission of a factor in the product. As the multi-indices $I_1, I_2, \ldots, I_h$ and $\Jc_1, \Jc_2, \ldots, \Jc_h$ are given, it follows that
\begin{gather*}
\sum_{\abs{\Krm} = k} \frac{{\rm d}^{\abs{\Krm}}}{{\rm d}x^\Krm} \pd{L}{u^\beta_\Krm}
= \sum_{\substack{1 \le r,s \le h \\ s \ne r}} \sum_{\sigma \in \Sfk_h} \delta^{\alpha_r}_\beta
\varepsilon_\sigma \Phi_{rs\sigma} u^{\alpha_s}_{I_r + I_s + \Jc_{\sigma(r)} + \Jc_{\sigma(s)}},
\end{gather*}
where the factor $\delta^K_{I_r + \Jc_{\sigma(r)}}$ on the right-hand side is omitted.

Fix values for $r$ and $s$; we shall show that
\begin{gather*}
\sum_{\sigma \in \Sfk_h} \delta^{\alpha_r}_\beta \varepsilon_\sigma \Phi_{rs\sigma} u^{\alpha_s}_{I_r + I_s + \Jc_{\sigma(r)} + \Jc_{\sigma(s)}} = 0 .
\end{gather*}
To see this, for each $\sigma \in \Sfk_h$ let $\sigmatil$ be the permutation given by
\begin{gather*}
\sigmatil(r) = \sigma(s) , \qquad \sigmatil(s) = \sigma(r) , \qquad \sigmatil(t) = \sigma(t) \qquad \text{for $1 \le t \le h$, $t \ne r,s$} .
\end{gather*}
We obtain in this way a partition of $\Sfk_h$ into equivalence classes of the form $\{\sigma,\sigmatil\}$, where each equivalence class contains exactly two elements because $r \ne s$. As $\Phi_{rs\sigmatil} = \Phi_{rs\sigma}$, $\Jc_{\sigmatil(r)} + \Jc_{\sigmatil(s)} = \Jc_{\sigma(r)} + \Jc_{\sigma(s)}$ and $\varepsilon_\sigma + \varepsilon_{\sigmatil} = 0$, the result follows.
\end{proof}

\begin{Corollary}The upper bound $p_k$ given in Theorem~{\rm \ref{poly}}, for the polynomial degree of a~Lag\-rangian in the highest order derivatives, is sharp if the number of independent variables~$n$ satisfies $n \ge p_k$.
\end{Corollary}
\begin{proof}Take $q=k$, and let $\Jc_1, \Jc_2, \ldots, \Jc_{p_k}$ be the distinct multi-indices of length $k$. If
\begin{gather*}
L = \begin{vmatrix}
u^1_{\Jc_1} & u^1_{\Jc_2} & \cdots & u^1_{\Jc_{p_k}} \\
u^2_{\Jc_1} & u^2_{\Jc_2} & \cdots & u^2_{\Jc_{p_k}} \\
\vdots & \cdots & \ddots & \vdots\\
u^{p_k}_{\Jc_1} & u^{p_k}_{\Jc_2} & \cdots & u^{p_k}_{\Jc_{p_k}}
\end{vmatrix}
\end{gather*}
(so that all the multi-indices $I_r$ are zero) then $L$ is the coefficient of a hyperaffine hyperform, so that it gives rise to projectable Euler--Lagrange equations by Theorem~\ref{Th1}.
\end{proof}

\section{Discussion}\label{section6}

The arguments above show that there is a correspondence between hyperaffine $(p_q,q)$ hyperforms and Lagrangian densities with projectable Euler--Lagrange equations. The correspondence is certainly not injective, even locally in a fixed coordinate system. For instance, with $m=k=2$ and $n=3$ the Lagrangian functions
\begin{gather*}
L_3 = \begin{vmatrix}
u_{xx} & u_{xy} & u_{yy} \\ v_{xx} & v_{xy} & v_{yy} \\ w_{xx} & w_{xy} & w_{yy}
\end{vmatrix} \!, \!\!\qquad
L_5 = w \begin{vmatrix}
u_{xx} & u_{xy} \\ v_{xy} & v_{yy}
\end{vmatrix} \!,\!\! \qquad
L_6 = w \begin{vmatrix}
u_{xx} & u_{xy} \\ v_{xx} & v_{xy}
\end{vmatrix}
= w \begin{vmatrix}
u_{xx} & u_{xy} & u_{yy} \\ v_{xx} & v_{xy} & v_{yy} \\ 0 & 0 & 1
\end{vmatrix}
\end{gather*}
are all hyperaffine; we have (temporarily omitting the symmetric product symbol $\odot$)
\begin{gather*}
L_3 {\rm d}x {\rm d}x \wedge {\rm d}x {\rm d}y \wedge {\rm d}y {\rm d}y = \bigl( u_{xx} {\rm d}x {\rm d}x + u_{xy} {\rm d}x {\rm d}y + u_{yy} {\rm d}y {\rm d}y \bigr)\\
\hphantom{L_3 {\rm d}x {\rm d}x \wedge {\rm d}x {\rm d}y \wedge {\rm d}y {\rm d}y =}{}
\wedge \bigl( v_{xx} {\rm d}x {\rm d}x + v_{xy} {\rm d}x {\rm d}y + v_{yy} {\rm d}y {\rm d}y \bigr) \\
\hphantom{L_3 {\rm d}x {\rm d}x \wedge {\rm d}x {\rm d}y \wedge {\rm d}y {\rm d}y =}{}
 \wedge \bigl( w_{xx} {\rm d}x {\rm d}x + w_{xy} {\rm d}x {\rm d}y + w_{yy} {\rm d}y {\rm d}y \bigr), \\
L_5 {\rm d}x \wedge {\rm d}y = \bigl( w (u_{xx} {\rm d}x + u_{xy} {\rm d}y) \bigr) \wedge \bigl( v_{xy} {\rm d}x + v_{yy} {\rm d}y \bigr),
\end{gather*}
but
\begin{gather*}
L_6 {\rm d}x \wedge {\rm d}y = \bigl( w (u_{xx} {\rm d}x + u_{xy} {\rm d}y \bigr) \wedge \bigl( v_{xx} {\rm d}x + v_{xy} {\rm d}y \bigr), \\
L_6 {\rm d}x {\rm d}x \wedge {\rm d}x {\rm d}y \wedge {\rm d}y {\rm d}y = \bigl( u_{xx} {\rm d}x {\rm d}x + u_{xy} {\rm d}x {\rm d}y + u_{yy} {\rm d}y {\rm d}y \bigr) \\
\hphantom{L_6 {\rm d}x {\rm d}x \wedge {\rm d}x {\rm d}y \wedge {\rm d}y {\rm d}y =}{} \wedge \bigl( v_{xx} {\rm d}x {\rm d}x + v_{xy} {\rm d}x {\rm d}y + v_{yy} {\rm d}y {\rm d}y \bigr) \wedge {\rm d}y {\rm d}y .
\end{gather*}

We do, however, make the conjecture that the correspondence is surjective: that is, that if a~Lagrangian density has projectable Euler--Lagrange equations then, in any coordinate system, its Lagrangian function must be a~sum of determinants of the form~\eqref{Det}, or of minors of such determinants with essentially the same format. One might clearly attempt to establish such a~conjecture by considering the homogeneous components of the Lagrangian, and it is certainly the case that the quadratic component satisfies the condition.

\begin{Proposition}If the Lagrangian density $\lambda$ on $J^k\pi$ has projectable Euler--Lagrange equations then the quadratic terms of the polynomial Lagrangian function $L$ are determinants with a hyperaffine structure.
\end{Proposition}
\begin{proof}Suppose the quadratic terms of $L$ are $A^{\Hrm_1 \Hrm_2}_{\alpha_1 \alpha_2} u^{\alpha_1}_{\Hrm_1} u^{\alpha_2}_{\Hrm_2}$ where $\abs{\Hrm_1} = \abs{\Hrm_2} = k$. Partition the set of quadratic terms according to the multi-index $H = \Hrm_1 + \Hrm_2$, and consider the terms
\begin{gather*}
\Psi_H = \sum_{\Hrm_1 + \Hrm_2 = H} A^{\Hrm_1 \Hrm_2}_{\alpha_1 \alpha_2} u^{\alpha_1}_{\Hrm_1} u^{\alpha_2}_{\Hrm_2}
\end{gather*}
in a single component of the partition. There must be at least two distinct terms; for if there were only a single term then there would have to be an index $i$ such that $\Hrm_1(i) = \Hrm_2(i) = k$, and then~\eqref{Eq6} would imply $A^{\Hrm_1 \Hrm_2}_{\alpha_1 \alpha_2} = 0$.

Choose, arbitrarily, one term $A^{\Krm_1 \Krm_2}_{\alpha_1 \alpha_2} u^{\alpha_1}_{\Krm_1} u^{\alpha_2}_{\Krm_2}$, so that~\eqref{Eq6} now gives
\begin{gather*}
A^{\Krm_1 \Krm_2}_{\alpha_1 \alpha_2} = \sum_{\substack{\Hrm_1 + \Hrm_2 = H \\ (\Hrm_1,\Hrm_2) \ne (\Krm_1,\Krm_2)}} - A^{\Hrm_1 \Hrm_2}_{\alpha_1 \alpha_2},
\end{gather*}
and hence
\begin{gather*}
\Psi_H = \sum_{\Hrm_1 + \Hrm_2 = H} A^{\Hrm_1 \Hrm_2}_{\alpha_1 \alpha_2}
\bigl( u^{\alpha_1}_{\Hrm_1} u^{\alpha_2}_{\Hrm_2} - u^{\alpha_1}_{\Krm_1} u^{\alpha_2}_{\Krm_2} \bigr),
\end{gather*}
so that each $\Psi_H$ is a sum of determinants. (The restriction $(\Hrm_1,\Hrm_2) \ne (\Krm_1,\Krm_2)$ is omitted from the latter sum because if $(\Hrm_1,\Hrm_2) = (\Krm_1,\Krm_2)$ then the term vanishes.)

To see that each determinant has a hyperaffine structure (that is, can be written in the form~\eqref{Det}) consider a single expression
\begin{gather}\label{2x2}
u^{\alpha_1}_{\Hrm_1} u^{\alpha_2}_{\Hrm_2} - u^{\alpha_1}_{\Krm_1} u^{\alpha_2}_{\Krm_2}
\end{gather}
and let $I_1$, $I_2$ be the multi-indices defined by
\begin{gather*}
I_1(i) = \min \{\Hrm_1(i), \Krm_1(i)\} , \qquad I_2(i) = \min \{\Hrm_2(i), \Krm_2(i)\}, \qquad 1 \le i \le m .
\end{gather*}
Consider any index $i$. Suppose $I_1(i) = \Hrm_1(i)$, so that $\Hrm_1(i) \le \Krm_1(i)$; then $I_2(i) = \Krm_2(i)$, for if not we would have $\Hrm_2(i) = I_2(i) < \Krm_2(i)$, contradicting $\Hrm_1(i) + \Hrm_2(i) = \Krm_1(i) + \Krm_2(i)$. If instead $I_1(i) < \Hrm_1(i)$ then $I_1(i) = \Krm_1(i)$, and a similar argument shows that $I_2(i) = \Hrm_2(i)$.

Now let $\Jc_1$, $\Jc_2$ be the multi-indices defined by $I_1 + \Jc_1 = \Hrm_1$, $I_2 + \Jc_2 = \Hrm_2$. Consider any index~$i$. If $I_1(i) = \Hrm_1(i)$ and $I_2(i) = \Krm_2(i)$ then
\begin{gather*}
I_1(i) + \Jc_2(i) = \Hrm_1(i) + \Hrm_2(i) - I_2(i) = \Krm_1(i) + \Krm_2(i) - I_2(i) = \Krm_1(i)
\end{gather*}
and
\begin{gather*}
I_2(i) + \Jc_1(i) = \Krm_2(i) + \Hrm_1(i) - I_1(i) = \Krm_2(i) ,
\end{gather*}
and a similar argument shows that these relations also hold when $I_1(i) = \Krm_1(i)$ and $I_2(i) = \Hrm_2(i)$. We therefore see that $I_1 + \Jc_2 = \Krm_1$ and $I_2 + \Jc_1 = \Krm_2$, so that expression~\eqref{2x2} can be written as the determinant
\begin{gather*}
\begin{vmatrix}
u^{\alpha_1}_{I_1 + \Jc_1} & u^{\alpha_1}_{I_1 + \Jc_2} \\ u^{\alpha_2}_{I_2 + \Jc_1} & u^{\alpha_2}_{I_2 + \Jc_2}
\end{vmatrix} .
\end{gather*}
To see that this is indeed an instance of determinant~\eqref{Det} (or one of its minors), we must finally check that $\abs{I_1} = \abs{I_2}$. Let $P = \{i\colon I_1(i) = \Hrm_1(i)\}$ and $Q = \{i \colon I_1(i) < \Hrm_1(i)\}$; then
\begin{align*}
\abs{I_1} - \abs{I_2} & = \sum_{i \in P} \bigl( I_1(i) - I_2(i) \bigr) + \sum_{i \in Q} \bigl( I_1(i) - I_2(i) \bigr) \\
& = \sum_{i \in P} \bigl( \Hrm_1(i) - \Krm_2(i) \bigr) + \sum_{i \in Q} \bigl( \Krm_1(i) - \Hrm_2(i) \bigr) \\
& = \sum_{i \in P} \bigl( \Hrm_1(i) - \Krm_2(i) \bigr) - \sum_{i \in P} \bigl( \Krm_1(i) - \Hrm_2(i) \bigr) ,
\end{align*}
and for any index $i$ we have $\Hrm_1(i) + \Hrm_2(i) = \Krm_1(i) + \Krm_2(i)$. We may therefore set $q = \abs{\Jc_1} = \abs{\Jc_2}$ so that $\abs{I_1} = \abs{I_2} = k-q$.
\end{proof}

A similar result for an arbitrary homogeneous component of~$L$ does, however, seem to be significantly more complicated to prove, and so work continues on the project.

\subsection*{Acknowledgements}

The author would like to thank the referees for their helpful suggestions regarding the presentation of some technical aspects of this work.

\pdfbookmark[1]{References}{ref}
\LastPageEnding


\begin{thebibliography}{99}
\footnotesize\itemsep=0pt

\bibitem{And89}
Anderson I.M., The variational bicomplex, {T}echnical report, Utah State
 University, 1989.

\bibitem{Car82}
Carmeli M., Classical fields: general relativity and gauge theory, \textit{A~Wiley-Interscience Publication}, John Wiley
 \& Sons, Inc., New York, 1982.

\bibitem{Kru90}
Krupka D., Variational sequences on finite order jet spaces, in Differential
 Geometry and its Applications ({B}rno, 1989), World Sci. Publ., Teaneck, NJ,
 1990, 236--254.

\bibitem{Krv86}
Krupkov\'a O., Lepagean {$2$}-forms in higher order {H}amiltonian mechanics.
 {I}.~{R}egularity, \textit{Arch. Math. (Brno)} \textbf{22} (1986), 97--120.

\bibitem{Krv97}
Krupkov\'a O., The geometry of ordinary variational equations, \href{https://doi.org/10.1007/BFb0093438}{\textit{Lecture
 Notes in Mathematics}}, Vol.~1678, Springer-Verlag, Berlin, 1997.

\bibitem{Olv82}
Olver P.J., Differential hyperforms~I, {U}niversity of Minnesota Mathematics
 Report 82-101, 1982, available at
 \url{http://www-users.math.umn.edu/~olver/a_/hyper.pdf}.

\bibitem{Olv83}
Olver P.J., Hyper-{J}acobians, determinantal ideals and weak solutions to
 variational problems, \href{https://doi.org/10.1017/S0308210500013020}{\textit{Proc. Roy. Soc. Edinburgh Sect.~A}} \textbf{95}
 (1983), 317--340.

\bibitem{PV01}
Palese M., Vitolo R., On a class of polynomial {L}agrangians, \textit{Rend.
 Circ. Mat. Palermo Suppl.} (2001), 147--159, \href{https://arxiv.org/abs/math-ph/0111019}{math-ph/0111019}.

\bibitem{Sau89}
Saunders D.J., The geometry of jet bundles, \href{https://doi.org/10.1017/CBO9780511526411}{\textit{London Mathematical Society
 Lecture Note Series}}, Vol.~142, Cambridge University Press, Cambridge, 1989.

\end{thebibliography}
\end{document}